\documentclass[12pt]{article}
\usepackage{amsmath, amssymb, amsfonts}
\usepackage{graphicx}
\usepackage{enumerate}
\usepackage{breqn}
\usepackage{caption,subcaption}
\usepackage[colorlinks=true,citecolor=black,linkcolor=black,urlcolor=blue]{hyperref}
\usepackage{amsthm}
\usepackage{cleveref}

\newtheorem{theorem}{Theorem}[section]
\newtheorem{lemma}[theorem]{Lemma}
\newtheorem{cor}[theorem]{Corollary}

\numberwithin{equation}{section}
\numberwithin{figure}{section}

\newcommand{\Z}{\mathbb{Z}}
\newcommand{\arxiv}[1]{\href{http://arxiv.org/abs/#1}{\texttt{arXiv:#1}}}

\title{\bf Enumeration of lozenge tilings of a hexagon with a maximal staircase and a unit triangle removed}
\author{Ranjan Rohatgi\\
\small Department of Mathematics\\[-0.8ex]
\small Indiana University\\[-0.8ex] 
\small Bloomington, IN, 47405\\
\small\tt rrohatgi@indiana.edu\\}
\date{}

\begin{document}

\maketitle

 \begin{abstract}
Proctor proved a formula for the number of lozenge tilings of a hexagon with side-lengths $a,b,c,a,b,c$ after removing a ``maximal staircase." Ciucu then presented a weighted version of Proctor's result.  Here we present weighted and unweighted formulas for a similar region which has an additional unit triangle removed.  We use Kuo's graphical condensation method to prove the results.  By applying the factorization theorem of Ciucu, we obtain a formula for the number of lozenge tilings of a hexagon with three holes on consecutive edges.
 \end{abstract}
 
  \section{Introduction}
 
  The triangular lattice is the tiling of the plane by unit equilateral triangles.  Without loss of generality, we assume that the lattice comprises horizontal lines, as well as lines whose angles of incidence to the horizontal lines is either 60 or 120 degrees.  A \textit{region} in the triangular lattice is any finite union of these unit triangles and a \textit{lozenge} is any union of two unit triangles which share an edge.  A \textit{lozenge tiling} of a region $R$ is any covering of all unit triangles in $R$ by non-overlapping lozenges.  It is clear that a region must be have the same number of upward-pointing unit triangles as downward-pointing ones to have any tilings at all, since a lozenge contains one unit triangle of each type.  We say that such a region is \textit{balanced}.  We can assign to any lozenge that could be used in a tiling a  weight, $w$, which is a positive real number.  An \textit{unweighted} region has all weights equal to 1.    
 
The weight of a lozenge tiling of $R$ is the product of all the weights of the lozenges used in the tiling.  We denote by $M(R)$ the \textit{matching generating function} of the region $R$, which is the sum of the weights of all tilings of $R$.  For an unweighted region, the matching generating function simply gives the number of tilings of the region.
 
 MacMahon's work in \cite{macmahon} proved that for a hexagonal region with side-lengths $a,b,c,a,b,c$, the number of lozenge tilings is given by the formula 
 \begin{equation}\label{macmahonthm}
\frac{H(a)H(b)H(c)H(a+b+c)}{H(a+b)H(a+c)H(b+c)},
 \end{equation}
where we define the hyper factorials $H(n)$ for positive integers $n$ by
\begin{equation*}
H(n):=0!1!\ldots (n-1)!
\end{equation*}

The simplicity of (\ref{macmahonthm}) has inspired many to look for generalizations or similar results.  Here is a result of Proctor which enumerates the lozenge tilings of a hexagon with a ``maximal staircase" removed, denoted $P_{a,b,c}$ (see Figure~\ref{proctor}).

\begin{theorem}[Proctor \cite{proctor}]\label{proctorthm}
For any non-negative integers $a,b,$ and $c$ with $a\leq b$, we have $$M(P_{a,b,c})=\prod_{i=1}^a \left[\prod_{j=1}^{b-a+1}\frac{c+i+j-1}{i+j-1}\prod_{j=b-a+2}^{b-a+i}\frac{2c+i+j-1}{i+j-1}\right],$$ where empty products are taken to be $1$.  Further, $M(P_{b+1,b,c})=M(P_{b,b,c}).$
\end{theorem}

\begin{figure}
\centering
\includegraphics[height=2.5in]{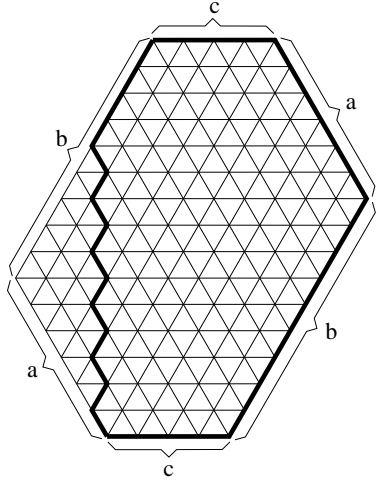}
\caption{The region $P_{a,b,c}$ for $a=6,b=9,\textrm{and } c=4.$}
\label{proctor}
\end{figure}

The following result of Ciucu provides a formula for the matching generating function of the same region, but with each of the vertical lozenges on the west side given weight $\frac{1}{2}$.  We denote this region by $P'_{a,b,c}$.  In Figure~\ref{proctorweight} (and throughout this paper), lozenges with ovals have weight $\frac{1}{2}$ while those without are unweighted.

 \begin{figure}[ht]
 \centering
 \includegraphics[height=2.5in]{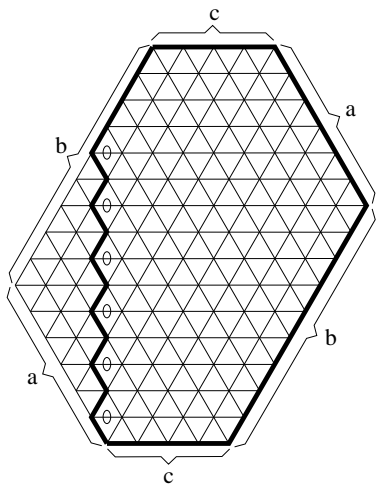}
 \caption{The region $P'_{a,b,c}$ has weighted lozenges along its west side.}
 \label{proctorweight}
 \end{figure}

\begin{theorem}[Ciucu \cite{ciucu05}]\label{ciucuthm}
For any non-negative integers $a,b,$ and $c$ with $a\leq b$ we have \[M(P'_{a,b,c})=\frac{M(P_{a,b,c})}{2^a}\cdot \prod_{i=1}^a \frac{2c+b-a+i}{c+b-a+i}.\]  As in \Cref{proctorthm}, $M(P'_{b+1,b,c})=M(P'_{b,b,c})$.
\end{theorem}

 \section{Main Result}

  Consider a hexagon with side-lengths $a+2,b,c+,a+1,b+1,c$, with both a maximal staircase and a single upward-pointing unit triangle removed, as in Figure~\ref{sandsprime}.  The removed unit triangle is the second from the bottom on the northeast side of the original hexagon.  We call such a region $S_{a,b,c}$.  We denote the corresponding weighted version, with all vertical lozenges on the west side with weight $\frac{1}{2}$, by $S'_{a,b,c}$.  We give the formulas for their matching generating functions below.  For ease of notation, define the Pochhammer symbol $(\alpha)_k$ for $k\in\Z$:
  
  \begin{displaymath}
(\alpha)_k:= \left\{
	\begin{array}{lr}
		\alpha(\alpha+1)\ldots (\alpha+k-1) & \textrm{ if } k>0,\\
		1 & \textrm{ if } k=0,\\ 
		1/(\alpha-1)(\alpha-2)\ldots (\alpha+k) & \textrm{ if } k<0.
	\end{array}
	\right.
\end{displaymath}

  \begin{figure}[ht]
 \centering
 \includegraphics[height=2.3in]{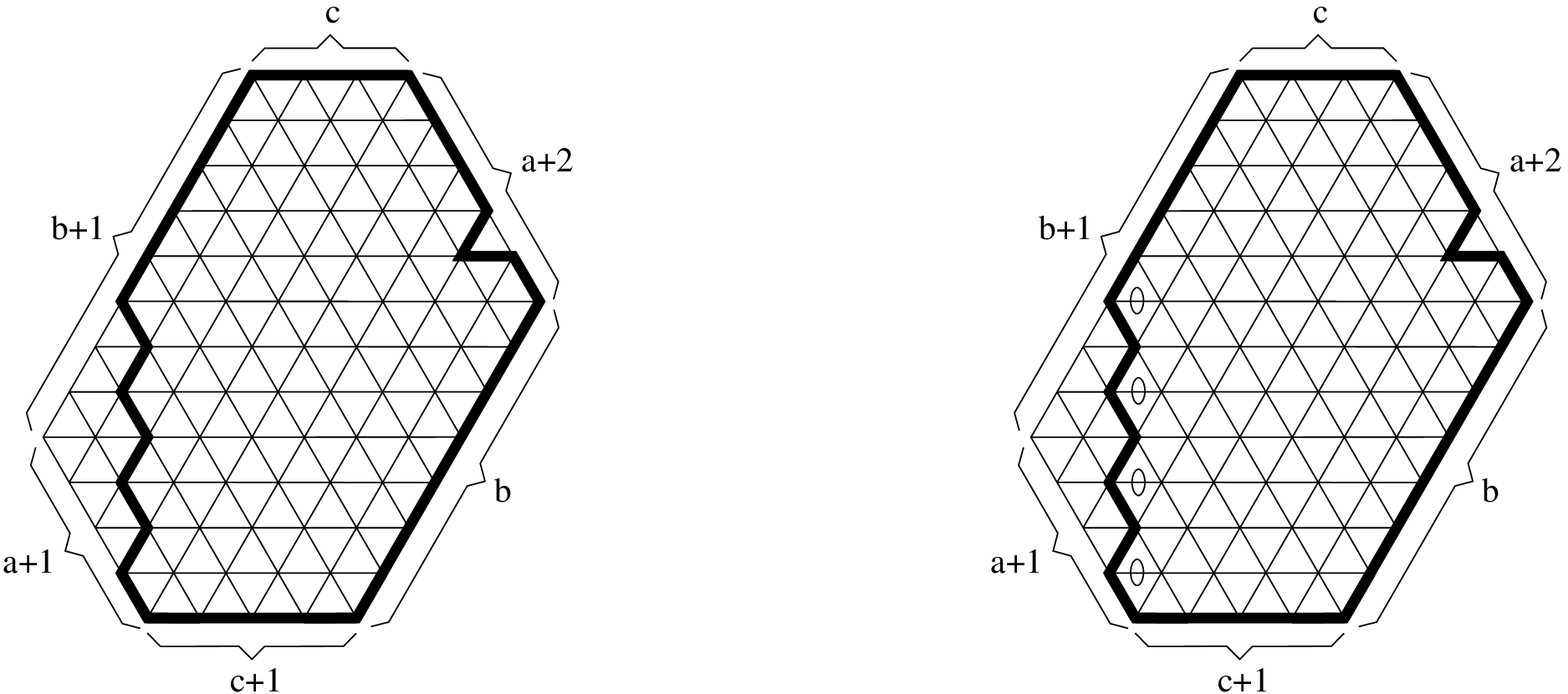}
 \caption{The regions $S_{a,b,c}$ and $S'_{a,b,c}$ with $a=3, b=7,$ and $c=3$.}
 \label{sandsprime}
 \end{figure}

  \begin{theorem} \label{mainthm}
  For any non-negative integers $a,b,$ and $c$ with $a\leq b$, we have 
  \begin{multline*}
  M(S_{a,b,c})=\dfrac{(c+a+2)_{b-a}(2c+b+3)_{a-1}[(2b-a+2)c+(b+1)(b+2)]}{(a+2)_{b-a+1}(b+3)_{a-1}}\\
  \times\prod_{i=1}^a\frac{(c+i)_{b-a+1}(2c+b-a+1+i)_{i-1}}{(i)_{b-a+1}(b-a+1+i)_{i-1}}. 
  \end{multline*}
  If $a=b+1$, $M(S_{b+1,b,c})=M(S_{b,b,c}).$
  \end{theorem}
  
  \begin{theorem} \label{weighted}
  For any non-negative integers $a,b,$ and $c$ with $a\leq b$, we have 
  \[ M(S'_{a,b,c})=\dfrac{M(S_{a,b,c})}{2^a}\cdot\dfrac{2c+b+2}{c+b+1}\cdot\prod_{i=1}^a\frac{2c+b+1-i}{c+b+1-i} .\]
  If $a=b+1$, $M(S'_{b+1,b,c})=M(S'_{b,b,c})$.
  \end{theorem}

  \section{Preliminaries}
  
  The \textit{dual graph} of a region $R$ is the graph comprising one vertex for each unit triangle in $R$.  Two vertices share an edge in the dual graph if and only if their corresponding unit triangles are edge-adjacent.  For regions on the triangular lattice, we've seen that each unit triangle is either pointing upwards or downwards - in particular, there are two types of unit triangles.  When creating the dual graph of $R$, we can recover this information by coloring the vertices corresponding to upward-pointing unit triangles one color, and those corresponding to downward-pointing unit triangles another.  The resulting graph is now bipartite, and lozenge tilings of a region $R$ are clearly in one-to-one correspondence with perfect matchings of the bipartite dual graph.  If a region has weighted lozenges, these correspond to weighted edges in the dual graph, and the matching generating functions of the region and graph still coincide.
  
 Translating our regions to their dual graphs allows us to make use of the graphical condensation method of Kuo, which provides an effective way to count perfect matchings (or matching generating functions) of bipartite graphs.  There are several versions; the one we will use is stated below.
  
  \begin{theorem}[Kuo condensation]\label{kuograph}
 Let $G=(V_1,V_2,E)$ be a plane bipartite graph with $|V_1|=|V_2|+1,$ and suppose that vertices $t,u,v,\textrm{and }w$ appear cyclically on a face of $G$.  If $t,u,v\in V_1$ and $w\in V_2$, then 
 \begin{multline*}M(G-u)M(G-\{t,v,w\})=\\ M(G-t)M(G-\{u,v,w\})+M(G-v)M(G-\{t,u,w\}).\end{multline*}
 \end{theorem}
  
 In \cite{ciuculai14}, Ciucu and Lai give conditions under which the matching generating function of a bipartite graph is the product of the matching generating function of two induced subgraphs.  We will need this result for some special cases in the proofs of \Cref{mainthm,weighted}.
  
    \begin{lemma}[Graph Splitting Lemma]\label{graphsplit}
Let $G=(V_1,V_2,E)$ be a bipartite graph.  Assume $H$ is an induced subgraph of $G$ that satisfies the following condition:
\begin{enumerate}[(i)]
\item (Separating condition) There are no edges of $G$ connecting a vertex in $V(H)\cap V_1$ and a vertex in $V(G-H)$.
\item (Balancing condition) $|V(H)\cap V_1| = |V(H)\cap V_2|$.
\end{enumerate}
Then $$M(G)=M(H)M(G-H).$$
\end{lemma}

  \section{Proofs of Theorems 2.1 and 2.2}
  
 We will prove \Cref{mainthm} via induction on $a$ using \Cref{kuograph}.  We will apply \Cref{kuograph} to the dual graph of the region $S_{a,b,c}$ without the unit triangle removed from the northeast side, as in Figure~\ref{kuoapplied}.  This region is unbalanced, as required by the theorem, and the locations of the vertices $t,u,v,$ and $w$ are given.  
 
  \begin{figure}[ht]
 \centering
 \includegraphics[height=2in]{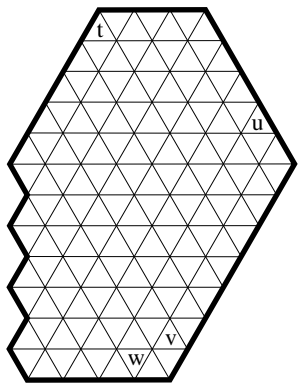}
 \caption{The region to which we apply Kuo condensation.}
 \label{kuoapplied}
 \end{figure}  

 Applying Kuo condensation to such a region gives us a recurrence involving six new regions.  They are shown in Figure~\ref{recurrencepic}.  In each subfigure, the triangles corresponding to removed vertices are labelled and any subsequently forced lozenges are shown.

 \begin{figure}
 \centering
 \includegraphics[height=7in]{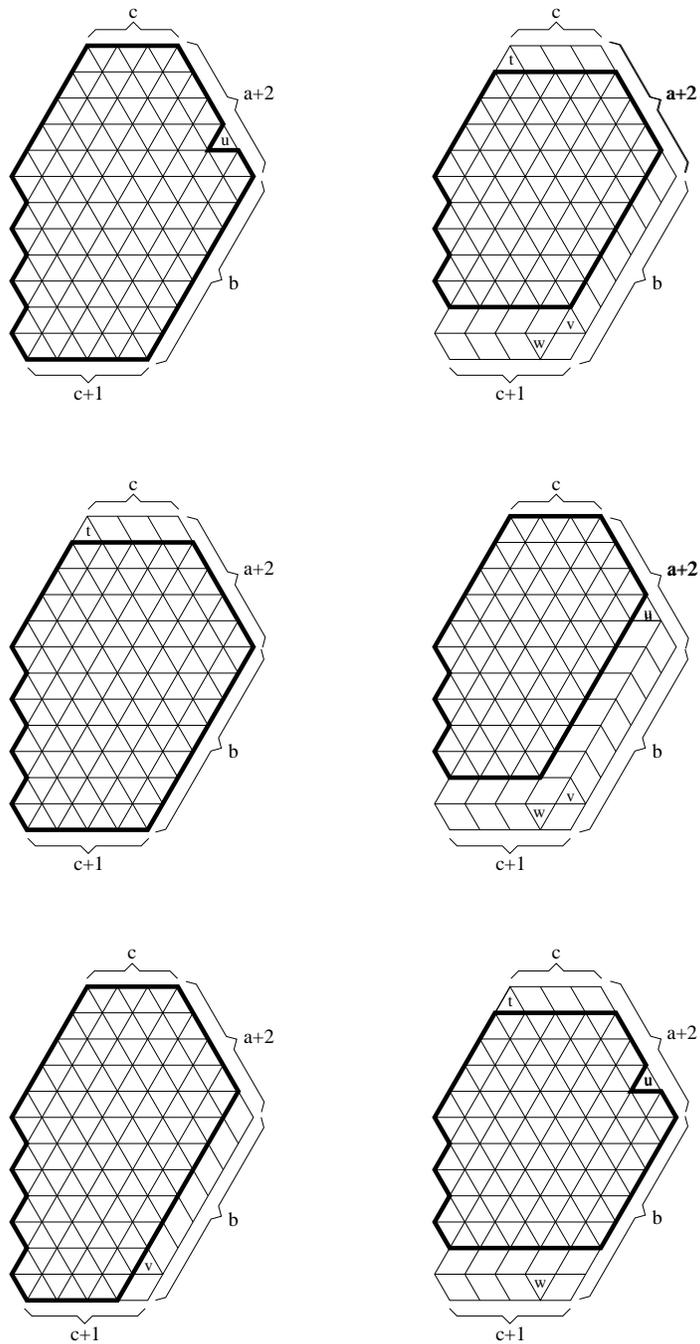}
 \caption{The six regions obtained after applying Kuo condensation.}
 \label{recurrencepic}
 \end{figure}

   Based on forcing, we see that 
 \begin{align*}
 M(G-u) &=  M(S_{a,b,c}),\\
 M(G-\{t,v,w\}) &=M(P_{a,b-1,c+1}),\\
 M(G-t) &=M(P_{a+1,b,c+1}),\\
 M(G-\{u,v,w\}) &=M(P_{a,b,c}),\\
 M(G-v) &=M(P_{a+1,b+1,c}),\textrm{ and }\\
 M(G-\{t,u,w\}) &=M(S_{a-1,b-2,c+1}).
 \end{align*}
  
  Therefore, we know 
  
  \begin{multline}\label{rec}
  M(S_{a,b,c})M(P_{a,b-1,c+1})=\\
   M(P_{a+1,b,c+1})M(P_{a,b,c}) + M(P_{a+1,b+1,c})M(S_{a-1,b-2,c+1}),
  \end{multline}
  
 as long as $a\geq 1$ and $b\geq 2$.  \Cref{proctorthm} explicitly gives the matching generating functions for the $P$-type regions, so \Cref{rec} is merely a recurrence between the matching generating functions of $S$-type regions.  We will prove \Cref{mainthm} by induction on $a$.  Therefore, it suffices to show that the claimed formula in \Cref{mainthm} satisfies \Cref{rec} and holds for $a=0$.  For completeness, we will also show that this formula holds when $b=0$ or $b=1$.
 
 If $a=0$, the formula from \Cref{mainthm} implies
 \begin{multline*}
 M(S_{0,b,c})=\\ \frac{(c+2)_b(2c+b+3)_{-1}[(2b+2)c+(b+1)(b+2)]}{2_{b+1}(b+3)_{-1}}={b+c+1 \choose b}.
 \end{multline*}
 
 On the other hand, if $a=0$, then the north-most row is forced and the resulting region is a hexagon with side-lengths $c+1,1,b,c+1,1,b$ as in Figure~\ref{a0}.  MacMahon's formula (\ref{macmahonthm}) verifies that the number of tilings of this hexagon is indeed ${b+c+1\choose b}$.
 
   \begin{figure}
 \centering
 \includegraphics[height=1.5in]{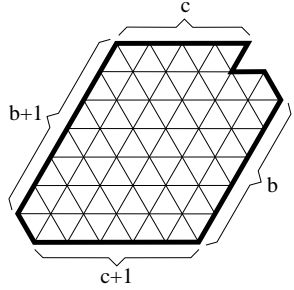}
 \caption{When $a=0$, it is clear that the north-most row is forced.}
 \label{a0}
 \end{figure}  
 
 We now need to check that \Cref{rec} holds.  First, we rewrite \Cref{proctorthm} as \[M(P_{a,b,c})=\prod_{i=1}^a\frac{(c+i)_{b-a+1}(2c+b-a+1+i)_{i-1}}{i_{b-a+i}}.\]
Using this formula for the $P$-type regions and the formula from \Cref{mainthm} for the $S$-type regions, \Cref{rec} becomes
 
 \begin{multline}\label{bigrec}
 \frac{(c+a+2)_{b-a}(2c+b+3)_{a-1}[(2b-a+2)c+(b+1)(b+2)]}{(a+2)_{b-a+1}(b+3)_{a-1}}\\
 \times\prod_{i=1}^a\frac{(c+i)_{b-a+i}(2c+b-a+1+i)_{i-1}}{i_{b-a+1}(b-a+1+i)_{i-1}}\\ \times\prod_{i=1}^a\frac{(c+1+i)_{b-a}(2c+b-a+2+i)_{i-1}}{i_{b-a+i+1}}\\
 =\prod_{i=1}^{a+1}\frac{(c+1+i)_{b-a}(2c+b-a+2+i)_{i-1}}{i_{b-a-1+i}}\\ \times\prod_{i=1}^a\frac{(c+i)_{b-a+1}(2c+b-a+1+i)_{i-1}}{i_{b-a+i}}\\
 +\frac{(c+a+2)_{b-a-1}(2c+b+3)_{a-2}[(2b-a-1)(c+1)+b(b-1)]}{(a+1)_{b-a}(b+1)_{a-2}}\\
 \times\prod_{i=1}^{a-1}\frac{(c+1+i)_{b-a}(2c+b-a+2+i)_{i-1}}{i_{b-a}(b-a+i)_{i-1}}\\ \times\prod_{i=1}^{a+1}\frac{(c+i)_{b-a+1}(2c+b-a+1+i)_{i-1}}{i_{b-a+i}}.
 \end{multline}

 Through straightforward algebraic manipulation, one can verify that \Cref{bigrec} is true.

 To complete the proof of \Cref{mainthm}, we need to show the result for $b=0,1$.  As $0\leq a\leq b$, the only remaining case is when $a=b=1$.  In this case, \Cref{mainthm} says that $M(S_{1,1,c})=(c+1)(c+2).$  
 
    \begin{figure}
 \centering
 \includegraphics[height=1.2in]{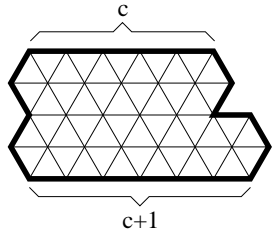}
 \caption{The $a=b=1$ case.}
 \label{ab1}
 \end{figure} 
 
 We get a region as in Figure~\ref{ab1} when $a=b=1.$  At this point we apply \Cref{graphsplit}, and we take $H$ to be the dual subgraph to the top two rows of $S_{1,1,c}$ (which makes $G-H$ the dual graph to the bottom two rows).  It is clear that $M(S_{1,1,c})$ is the product of the matching generating functions of two hexagons - one with side-lengths $c,1,1,c,1,1$ and the other $c+1,1,1,c+1,1,1.$  Using (\ref{macmahonthm}), we see that the matching generating functions are $c+1$ and $c+2$ respectively. \qed
 
 The proof of \Cref{weighted} is similar to that of \Cref{mainthm}.  \Cref{kuograph} is applied in exactly the same way, yielding a recurrence identical to that of \Cref{rec} but with $S$ and $P$ replaced by $S'$ and $P'$, respectively.  Verifying that the formula in \Cref{weighted} satisfies the new recurrence is done similarly, as are the few special cases.
  
\section{Symmetric triply-dented hexagons}
  By symmetrizing our region along the ``maximal staircase" we obtain a symmetric triply-dented hexagon, $STDH_{a,b,c}$ as in Figure~\ref{stdh} below.  Notice that the removed triangle on the north edge is centrally located and that the removed unit triangles on the northwest and northeast edges are at distance one away from the west and east corners of the region, respectively.
  
  \begin{figure}
  \centering
  \includegraphics[height=2in]{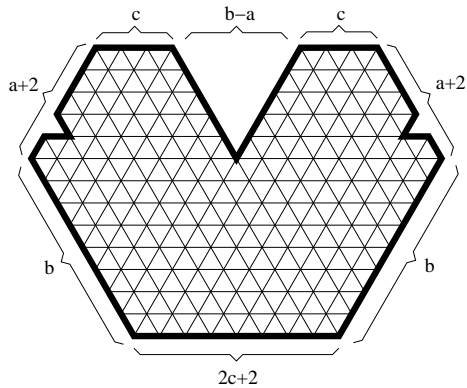}
 \caption{$STDH_{3,8,3}$.}
 \label{stdh}
 \end{figure}

  Using Ciucu's factorization theorem from \cite{ciucu97}, we get the following result in the vein of Lai \cite{lai} and Eisenk{\"o}lbl \cite{eisen}.
  
  \begin{cor}
  For non-negative integers $a,b,$ and $c$ with $a\leq b$ we have 
  \[M(STDH_{a,b,c})=2^{a+1} M(S'_{a,b,c})M(S_{a,b-1,c}).\]
  \end{cor}
  
  \begin{proof}
  To apply the factorization theorem we must cut the region into the two subregions shown in Figure~\ref{stdhfac}.

    \begin{figure}
 \centering
 \includegraphics[height=2in]{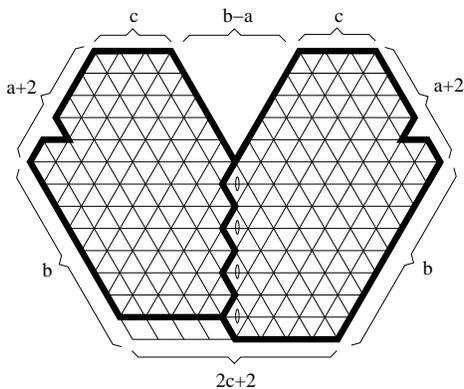}
 \caption{Applying the factorization theorem to $STDH_{3,8,3}$.}
 \label{stdhfac}
 \end{figure}
 After forcing, we have one region of type $S$ and another of type $S'$.  The result follows immediately.
  \end{proof}

 \end{document}